\documentclass{amsart}
\makeatletter
\@ifundefined{ifjocformat}{\newif\ifjocformat}{}
\makeatother
\ifjocformat\else
\usepackage{palatino}
\usepackage{fullpage}
\fi
\usepackage{amsmath,amsthm}
\usepackage{amssymb,amsfonts,amscd}
\usepackage{mathtools,thmtools,thm-restate}
\ifjocformat\else
\usepackage{microtype}
\fi
\usepackage{needspace}
\usepackage{etoolbox}
\AtBeginEnvironment{thebibliography}{\interlinepenalty=10000}
\usepackage[T1]{fontenc}
\usepackage{enumitem}
\usepackage{tikz}
\usepackage{hyperref}
\usepackage[capitalise,noabbrev]{cleveref}
\usepackage{placeins}
\usepackage{caption}
\hypersetup{
  colorlinks=true,
  linkcolor=blue,
  citecolor=red,
  urlcolor=blue
}

\numberwithin{equation}{section}
\setlist[itemize]{leftmargin=1.8em,itemsep=0.2em,topsep=0.25em}
\setlist[enumerate]{leftmargin=2em,itemsep=0.2em,topsep=0.25em}
\allowdisplaybreaks[1]
\hypersetup{
  pdftitle={First-Derivative Chromatic Symmetric Reconstruction for Proper Trees},
  pdfauthor={Saad A. Awan}
}

\newtheorem{lemma}{Lemma}[section]
\newtheorem{proposition}[lemma]{Proposition}

\newcommand{\Q}{\mathbb Q}
\newcommand{\Comp}{\operatorname{Comp}}
\newcommand{\one}[1]{\mathbf 1_{\{#1\}}}

\title[First-Derivative Reconstruction]{First-Derivative Chromatic Symmetric\\ Reconstruction for Proper Trees}
\author{Saad A. Awan}
\address{Department of Mathematics, University of Kansas, Lawrence, Kansas 66045}
\email{saad\_awan@ku.edu}

\date{October 2026}


\begin{document}
\ifjocformat
\begin{frontmatter}
\input{metadata-joc}
\fi
\begin{abstract}
Richard Stanley asked whether every tree $T$ is determined up to isomorphism by its chromatic symmetric function $X_T$. We approach this open problem by regarding $X_T$ as a polynomial in the power-sum symmetric functions $p_1, p_2, \dots$ and studying the invariant $\Phi_T = \left.\frac{\partial {X_T}}{\partial p_1}\right\rvert_{p_1 = 0}$. We prove that $\Phi_T$ distinguishes, among proper trees, every proper tree in which any non-leaf vertex with two or more non-leaf neighbors has a number of leaf neighbors occurring at no other non-leaf vertex. We give an equivalent formulation of Stanley's question in terms of $\Phi_T$. Finally, we show that counts of spanning forests with at most $t$ edges cannot distinguish all $k$-vertex trees when $t<\left\lfloor k/2\right\rfloor$, for every $k\ge4$.

\end{abstract}

\ifjocformat
\tableofcontents
\end{frontmatter}
\else
\maketitle
\tableofcontents
\fi

\section*{Introduction}

For a finite simple graph $G=(V,E)$, let
$\operatorname{PCol}(G)$ be the set of all \emph{proper} colorings
$\kappa\colon V\to\mathbb{N}^+$, meaning $\kappa(v) \neq \kappa(w)$ whenever $v$ and $w$ are adjacent. Stanley's chromatic symmetric function is defined
\[
X_G= X_G(x_1, x_2, \dots)=\sum_{\kappa\in\operatorname{PCol}(G)}\prod_{v\in V}x_{\kappa(v)},
\]
where the sum is over all proper colorings $\kappa$, and $\{x_1, x_2, \dots\}$ is a countably infinite set of commuting indeterminates. Under any permutation of $\{x_i\}$ this function remains unchanged, and thus is a symmetric function.
Stanley~\cite{Stanley1995} introduced this invariant and asked whether for \emph{trees} it is a complete isomorphism invariant. Equivalently, let $T, T'$ be two trees. If $X_T = X_{T'}$, does it follow that $T \cong T'$? The question remains open.

Stanley's question has been answered in the affirmative for several classes of trees.
Martin, Morin, and Wagner~\cite{Mart2007} showed that all \emph{spiders} and certain
\emph{caterpillars} are determined by their chromatic symmetric function.
They also showed that nonisomorphic \emph{squids}, which are unicyclic
graphs, have different chromatic symmetric functions.
Aliste-Prieto and Zamora~\cite{AZ2014} answered Stanley's question positively for \emph{proper
caterpillars}, and Loebl and Sereni~\cite{LS2019} removed the properness
hypothesis, settling the case of all caterpillars.
Huryn and Chmutov~\cite{HC2020} generalised the spider result to
\emph{$2$-spiders}.
More recently, Aliste-Prieto, de Mier, Orellana, and Zamora~\cite{APdMOZ2023} used the star
basis and their deletion--near-contraction relation to show $X_T$ distinguishes
proper trees of diameter at most five.
Gonzalez, Orellana, and Tomba~\cite{GOT2025} subsequently removed the properness
hypothesis, reconstructing every tree of diameter at most five directly
from its chromatic symmetric function.

One particular approach which has proved to be productive has been to
differentiate the chromatic symmetric function with respect to the
power-sum variables $p_j$. Although the idea originated with Stanley, it was Wang, Yu, and Zhang~\cite{WYZ2024} who showed that trees with
exactly two vertices of degree at least three are determined by their
chromatic symmetric function using this approach.
Throughout, we use the first derivative $\frac{\partial X_T}{\partial p_1}$.

Recall that the \emph{degree} of a vertex is the number of vertices adjacent to it. We define a \emph{leaf} to be a vertex with degree $1$. We define a tree $T$ with at least three vertices to be \emph{proper} when every non-leaf vertex is adjacent to a leaf.  Let $c_1,\ldots,c_k$ be the non-leaf vertices of $T$. We assign a weight of $1$ to each $c_i$, let us denote it $w(i)$.
For each $c_i$, contract every edge adjacent to a leaf vertex, and for each contraction, increase $w(i)$ by $1$. Doing this for all $c_i$, we are left with a weighted tree which we call the \emph{weighted skeleton}. The
\emph{weighted skeleton} $(S,\mathbf w)$ is a pair with $S = (V,E)$, a labeled tree on $k$ vertices and $\mathbf{w} = (w_1, w_2, \dots, w_k)$ denoting the weight for each labeled vertex. We call
$(S,\mathbf w)$ \emph{multiplicity-isolated} if every weight that occurs more than once occurs only at leaves of $S$. Every proper tree $T$ can be reconstructed from its weighted skeleton and vice versa; refer to Figure~\ref{fig:weighted-skeleton}.

Zeng~\cite{Zeng2026} proved that $X_T$
distinguishes every proper tree $T$ whose weighted skeleton is multiplicity-isolated. We show that in fact, every proper tree $T$ whose weighted skeleton is multiplicity-isolated can be reconstructed using a weaker invariant.

Let $\pi$ be the ring
endomorphism of $\Q[p_1,p_2,\ldots]$ defined by
\[
\pi(p_1)=0,
\qquad
\pi(p_k)=p_k\quad(k\ge2).
\]
and let
\[
\Phi_G=\pi\left(\frac{\partial X_G}{\partial p_1}\right) = \left.\frac{\partial X_G}{\partial p_1} \right\rvert_{p_{1} = 0}.
\]

We state our main results.
\begin{restatable}{theorem}{palm}\label{thm:palm}
If the weighted skeleton of a proper tree $T$ is multiplicity-isolated and
$T'$ is any proper tree, then
\[
\Phi_T=\Phi_{T'}\implies T\cong T'.
\]
\end{restatable}
Thus the same class of proper trees can be distinguished, among proper trees, using $\Phi_T$ in place of $X_T$~\cite{Zeng2026}.

Our second theorem asserts an equivalent formulation of Stanley's question. For $q\ge2$, the \emph{uniform leaf extension} of $G$ is the graph $C_q(G)$ obtained by attaching $q-1$ new leaves to every vertex $v \in V(G)$.

\begin{restatable}{theorem}{Cq}\label{thm:Cq}
For fixed $q\ge2$ and all nonempty finite connected graphs $G,H$,
\[
\Phi_{C_q(G)}=\Phi_{C_q(H)}
\quad\Longleftrightarrow\quad
X_G=X_H.
\]
\end{restatable}
Since all trees are finite connected graphs, the consequence of this result is that the chromatic symmetric function distinguishes trees if and only if the invariant $\Phi_{C_q(T)}$ does.

Before we state our third theorem, we define some preliminary notation.
Let $\mathcal T_k$ be the set of isomorphism classes of $k$-vertex trees.
For $S\in\mathcal T_k$ and $F\subseteq E(S)$, the \emph{component type}
of $F$ is the integer partition of $k$ formed by the orders of the
connected components of the spanning forest $(V(S),F)$.
For a partition $\mu\vdash k$, let $N_\mu(S)$ be the number of edge sets
$F\subseteq E(S)$ of component type $\mu$.
Write $\ell(\mu)$ for the number of parts of $\mu$. For each nonnegative
integer $t$, define
\[
\mathcal P_{k,t}
=\{\mu\vdash k\mid \ell(\mu)\ge k-t\}.
\]
We define the function
\[
\sigma_t\colon\mathcal T_k\longrightarrow
\mathbb Z_{\ge0}^{\mathcal P_{k,t}},
\qquad
\sigma_t(S)=\bigl(N_\mu(S)\bigr)_{\mu\in\mathcal P_{k,t}}.
\]
Since $(V(S),F)$ has $k-|F|$ connected components, the truncation
$\sigma_t(S)$ records the number of edge sets of each component type
among those containing at most $t$ edges.

Let
\[
\tau(k)=
\min\Bigl(
\{t\in\mathbb Z_{\ge0}\mid
\sigma_t\text{ is injective on }\mathcal T_k\}\cup\{\infty\}
\Bigr).
\]
Thus $\tau(k)$ is finite exactly when some truncation distinguishes all
$k$-vertex trees up to isomorphism.  A larger value means that more
edge-set layers are required.  A lower bound on $\tau(k)$ shows that no
shallower truncation can distinguish every tree. We demonstrate with an example. 

For $k=5$, let $A$, $B$, and $C$ be the three trees on five vertices, up to isomorphism, labeled as in Figure~\ref{fig:five-vertex-truncations}.

\begin{figure}[ht]
\centering
\begin{tikzpicture}[
  vertex/.style={circle,draw,fill=white,minimum size=5mm,inner sep=0pt}
]
\begin{scope}
\node[vertex] (a5) at (0,0) {$5$};
\node[vertex] (a1) at (0.8,0) {$1$};
\node[vertex] (a2) at (1.6,0) {$2$};
\node[vertex] (a3) at (2.4,0) {$3$};
\node[vertex] (a4) at (3.2,0) {$4$};
\draw (a5)--(a1)--(a2)--(a3)--(a4);
\node at (1.6,-1.25) {$A$};
\end{scope}
\begin{scope}[xshift=4.5cm]
\node[vertex] (b1) at (0,0) {$1$};
\node[vertex] (b2) at (0.9,0) {$2$};
\node[vertex] (b3) at (1.8,0) {$3$};
\node[vertex] (b4) at (2.7,0) {$4$};
\node[vertex] (b5) at (0.9,0.85) {$5$};
\draw (b1)--(b2)--(b3)--(b4);
\draw (b2)--(b5);
\node at (1.35,-1.25) {$B$};
\end{scope}
\begin{scope}[xshift=8.5cm]
\node[vertex] (c2) at (0.9,0) {$2$};
\node[vertex] (c1) at (0,0.65) {$1$};
\node[vertex] (c3) at (1.8,0.65) {$3$};
\node[vertex] (c4) at (0,-0.65) {$4$};
\node[vertex] (c5) at (1.8,-0.65) {$5$};
\draw (c2)--(c1) (c2)--(c3) (c2)--(c4) (c2)--(c5);
\node at (0.9,-1.25) {$C$};
\end{scope}
\end{tikzpicture}
\caption{The three trees on five vertices up to isomorphism.}
\label{fig:five-vertex-truncations}
\end{figure}
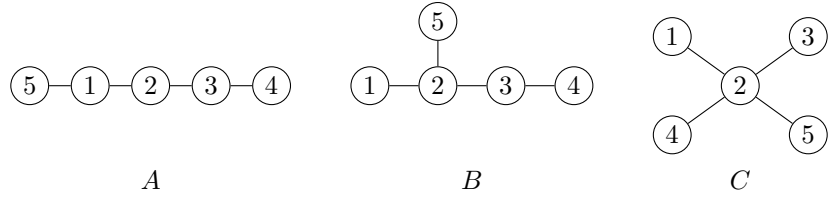

To compute $\sigma_t$, we enumerate the component types of all edge sets with $t$ or fewer edges. For instance, the components of
$F=\{12,23\}$ are
\[
\{1,2,3\},\qquad\{4\},\qquad\{5\},
\]
so the edge set $F$ contributes to $N_{(3,1,1)}(A)$, and $\lvert F \rvert = k - \ell(\mu) = 5 - 3 = 2$. The counts through two selected
edges are
\[
\begin{array}{c|c|ccc}
\lvert F \rvert & \mu
& N_\mu(A) & N_\mu(B) & N_\mu(C) \\
\hline
0 & (1,1,1,1,1) & 1 & 1 & 1 \\
1 & (2,1,1,1)   & 4 & 4 & 4 \\
2 & (3,1,1)     & 3 & 4 & 6 \\
2 & (2,2,1)     & 3 & 2 & 0
\end{array}
\]
The first two rows agree hence $\sigma_1$ cannot distinguish these trees. On the other hand, $\sigma_2$ distinguishes all three trees, and
\[
\tau(5)=2=\left\lfloor\frac{5}{2}\right\rfloor.
\]

\begin{restatable}{theorem}{counting}\label{thm:counting}
For every $k\ge4$,
\[
\tau(k)\ge\left\lfloor\frac{k}{2}\right\rfloor.
\]
Moreover, for every fixed $t\ge0$,
\[
\frac{
\bigl|\{T\in\mathcal T_k\mid
\sigma_t^{-1}(\sigma_t(T))=\{T\}\}\bigr|
}{
|\mathcal T_k|
}
\longrightarrow0
\qquad(k\to\infty).
\]
\end{restatable}
Theorem~\ref{thm:counting} gives a lower bound on $\tau(k)$ and asserts that,
for any fixed $t$, the proportion of trees uniquely determined by
$\sigma_t$ tends to zero as the number of vertices grows.

Section~\ref{sec:formalization} records the scope of the accompanying Lean development.

\section{Background}\label{sec:background}

Throughout we assume familiarity with standard definitions and notations of graph theory; see, e.g., ~\cite{diest2025}. The symbols
$\mathbb{N}^+$ and $\mathbb N$ denote the positive and nonnegative integers,
respectively, and $[k]=\{1,\ldots,k\}$.  All graphs are finite, simple, and
undirected.  For a proposition $P$, we define the function
\[
\mathbf{1}_{P} = 
\begin{cases}
    1 & \text{if } P \text{ is true}\\
    0 & \text{if } P \text{ is false}
\end{cases}
\]

\subsection{Graphs and trees}\label{subsec:graphs}

A graph is written $G=(V,E)$, where $V=V(G)$ is its vertex set and
$E=E(G)$ is its edge set.  The \emph{order} of $G$ is $|V(G)|$.  Two vertices are
\emph{adjacent} if they are the endpoints of a common edge.  The
\emph{degree} $\deg_G(v)$ of $v$ is the number of edges incident to $v$.
A vertex of degree $0$ is \emph{isolated}, and a vertex of degree $1$ is a
\emph{leaf}.  For $v\in V(G)$, the graph $G-v$ is obtained by deleting $v$
and all edges incident to it.
For $F\subseteq E(G)$, let $\Comp_G(F)$ be the set of vertex sets of the
components of $(V(G),F)$. A \emph{star} is a tree having a vertex adjacent to every other vertex.  A \emph{rooted star} specifies one such vertex as its \emph{center}.

A \emph{weighted graph} is a pair $(G,\mathbf w)$ consisting of a graph
$G$ and a weight function $\mathbf w\colon V(G)\to\mathbb{N}^+$.  For
$C\subseteq V(G)$, extend the weight function by
\[
\mathbf w(C)=\sum_{v\in C}\mathbf w(v).
\]
An isomorphism $\varphi\colon(G,\mathbf w)\to(H,\mathbf u)$ of weighted
graphs is a graph isomorphism satisfying
\[
\mathbf u(\varphi(v))=\mathbf w(v)
\]
for every $v\in V(G)$.

\subsection{Partitions and power-sum symmetric functions}
\label{subsec:partitions}

Let $x_1,x_2,\ldots$ be commuting indeterminates.  The ring
$\Lambda_{\Q}$ of symmetric functions is the graded $\Q$-algebra of
bounded-degree formal power series in these variables that are invariant
under every finite permutation of the variables.

For $n>0$, a \emph{partition} of $n$, written $\lambda\vdash n$, is a weakly
decreasing list $\lambda=(\lambda_1,\ldots,\lambda_\ell)$ of positive
integers with sum $n$.  Its number of parts is its \emph{length}, denoted
$\ell(\lambda)$.  The empty partition is the unique partition of $0$ and
has length $0$.  Let
\[
p_r=\sum_{i\ge1}x_i^r,
\qquad
p_\lambda=p_{(\lambda_1, \dots, \lambda_l)}=p_{\lambda_1}\cdots p_{\lambda_\ell}
\]
Over $\Q$, the polynomials $p_1,p_2,\ldots$ are algebraically independent
and
\[
\Lambda_{\Q}=\Q[p_1,p_2,\ldots]
\]
The polynomials $p_\lambda$ form a vector-space basis of $\Lambda_{\Q}$.
Every power-sum monomial is $p_\lambda$ for a unique partition $\lambda$.
Its degree is $|\lambda|$ and its number of power-sum factors is
$\ell(\lambda)$.  For a power-sum monomial $M$, write $[M]f$ for its
coefficient in $f$, and define
\[
[f]_r=\sum_{\ell(\lambda)=r}([p_\lambda]f)p_\lambda.
\]
This is the \emph{$r$-factor layer} of $f$.

\subsection{The chromatic symmetric function and its first derivative}
\label{subsec:csf}

A \emph{coloring} of $G=(V,E)$ is a map
$\kappa\colon V\to\mathbb{N}^+$.  It is \emph{proper} if
$\kappa(u)\ne\kappa(v)$ whenever $uv\in E$.  Let
$\operatorname{PCol}(G)$ denote the set of proper colorings of $G$.  The
\emph{chromatic symmetric function} of $G$ is
\begin{equation}\label{eq:csf-definition}
X_G
=
\sum_{\kappa\in\operatorname{PCol}(G)}
\prod_{v\in V}x_{\kappa(v)}.
\end{equation}
It is a homogeneous symmetric function of degree $|V(G)|$.

For $A\subseteq E(G)$, let $G_A=(V(G),A)$, and let $\lambda_G(A)$ be the
partition formed by the orders of its connected components.  We call
$\lambda_G(A)$ the \emph{type} of $A$.  Stanley's~\cite[Theorem~2.5]{Stanley1995} power-sum
expansion is
\begin{equation}\label{eq:stanley}
X_G=\sum_{A\subseteq E(G)}(-1)^{|A|}p_{\lambda_G(A)}.
\end{equation}
Since $p_1,p_2,\ldots$ are algebraically independent, this is the unique
expression for $X_G$ as a polynomial in the power sums.  Hence the
\emph{first power-sum derivative}
\[
\frac{\partial X_G}{\partial p_1}
\]
is unambiguously defined.  Let $\pi$ be the ring endomorphism defined by
$\pi(p_1)=0$ and $\pi(p_r)=p_r$ for $r\ge2$. Then we define
\begin{equation}\label{eq:phi-definition}
\Phi_G = \pi \left(\frac{\partial X_G}{\partial p_1} \right)
=\left.\frac{\partial X_G}{\partial p_1}\right|_{p_1=0}.
\end{equation}

Stanley~\cite[Corollary~2.12(a)]{Stanley1995} proved that for every
graph $G$,
\begin{equation}\label{eq:stanley-derivative}
\frac{\partial X_G}{\partial p_1}=\sum_{v\in V(G)}X_{G-v}.
\end{equation}

\begin{lemma}\label{lem:pi}
For every graph $G$,
\[
\pi X_G=
\sum_{\substack{A\subseteq E(G)\\
                 (V(G),A)\text{ has no isolated vertex}}}
(-1)^{|A|}p_{\lambda_G(A)}.
\]
\end{lemma}

\begin{proof}
Since $\pi$ is a ring homomorphism, applying it to \eqref{eq:stanley}
gives
\[
\pi X_G=\sum_{A\subseteq E(G)}(-1)^{|A|}\,\pi p_{\lambda_G(A)}.
\]
For $A\subseteq E(G)$, the monomial $p_{\lambda_G(A)}$ is the product of
$p_{|C|}$ over the components $C$ of $G_A$.  Since $\pi$ sends $p_1$ to
$0$ and fixes $p_r$ for $r\ge2$, this product is unchanged by $\pi$ when
every component has at least two vertices, and is sent to $0$ when some
component consists of a single vertex, that is, when $G_A$ has an
isolated vertex.  Hence exactly the edge sets $A$ for which $G_A$ has no
isolated vertex remain in the sum.
\end{proof}

\subsection{Proper trees and weighted skeletons}
\label{subsec:skeletons}

An edge of a tree is \emph{internal} if both endpoints are non-leaves.  Let
$T$ be a proper tree of order $n$, and let $c_1,\ldots,c_k$ be its non-leaf
vertices.  Delete every internal edge.  For each $i$, regard the component
containing $c_i$ as a star rooted at $c_i$, and let $w_i$ be its order.
These rooted stars are the \emph{leaf components} of $T$.

The \emph{weighted skeleton} of $T$ is the weighted tree $(S,\mathbf w)$
with vertex set $[k]$, weight $\mathbf w(i)=w_i$, and edge $ij$ whenever
$c_ic_j$ is an internal edge of $T$.  Thus
\[
n=\sum_{i=1}^k w_i.
\]
Conversely, $T$ is recovered from $(S,\mathbf w)$ by retaining the edges of
$S$ and attaching $\mathbf w(i)-1$ new leaves to each vertex $i$.

For $F\subseteq E(S)$, let $\Comp(F)$ be the set of vertex sets of the
components of the spanning forest $([k],F)$.  For $C\in\Comp(F)$, its order
is $|C|$ and its total weight is
\[
\mathbf w(C)=\sum_{i\in C}w_i.
\]
The weighted skeleton is \emph{multiplicity-isolated} if every weight that
occurs at more than one vertex occurs only at leaves of $S$.
In Figure~\ref{fig:weighted-skeleton}, the repeated weight $2$ occurs only at the two leaves of $S$.
\begin{figure}[!htb]
\centering
\begin{tikzpicture}[
    scale=0.8,
  transform shape,
  center/.style={circle,fill=black,inner sep=2.2pt},
  leaf/.style={circle,draw,fill=white,inner sep=2.2pt},
  weight/.style={circle,draw,minimum size=7mm,inner sep=0pt}
]
\begin{scope}
\node[center,label=below:$c_1$] (c1) at (0,0) {};
\node[center,label=below:$c_2$] (c2) at (1.8,0) {};
\node[center,label=below:$c_3$] (c3) at (3.6,0) {};
\node[leaf] (a1) at (0,0.9) {};
\node[leaf] (b1) at (1.35,0.9) {};
\node[leaf] (b2) at (2.25,0.9) {};
\node[leaf] (d1) at (3.6,0.9) {};
\draw (c1)--(c2)--(c3);
\draw (c1)--(a1) (c2)--(b1) (c2)--(b2) (c3)--(d1);
\draw[dashed,rounded corners] (-0.38,-0.42) rectangle (0.38,1.25);
\draw[dashed,rounded corners] (1.02,-0.42) rectangle (2.58,1.25);
\draw[dashed,rounded corners] (3.22,-0.42) rectangle (3.98,1.25);
\node at (1.8,-1.05) {$T$};
\end{scope}
\begin{scope}[xshift=6cm]
\node[weight,label=below:$1$] (s1) at (0,0) {$2$};
\node[weight,label=below:$2$] (s2) at (1.4,0) {$3$};
\node[weight,label=below:$3$] (s3) at (2.8,0) {$2$};
\draw (s1)--(s2)--(s3);
\node at (1.4,-1.05) {$(S,\mathbf w)$};
\end{scope}
\end{tikzpicture}
\captionsetup{font=footnotesize}
\caption{A proper tree and its weighted skeleton.}
\label{fig:weighted-skeleton}
\end{figure}
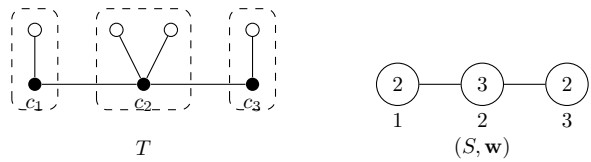

\FloatBarrier

\section{The First-Derivative Invariant}\label{sec:first-derivative}
In this section we prove Theorem~\ref{thm:palm}.

Fix a proper tree $T$ with weighted skeleton $(S,\mathbf w)$.
\begin{proposition}[The Master Formula]\label{prop:master}
Let $\varepsilon=(-1)^{n-k-1}$.  Then
\begin{equation}\label{eq:master}
\begin{aligned}
\Phi_T
&=\varepsilon\,\pi\Biggl(
\sum_{F\subseteq E(S)}(-1)^{|F|}
\sum_{C_0\in\Comp(F)}
(\mathbf w(C_0)-|C_0|)p_{\mathbf w(C_0)-1}\\
&\hspace{4em}{}\times
\prod_{C\in\Comp(F)\setminus\{C_0\}}p_{\mathbf w(C)}
\Biggr).
\end{aligned}
\end{equation}
In particular, every nonzero summand after specialization has degree $n-1$
and $k-|F|$ power-sum factors.
\end{proposition}

\begin{proof}
By Stanley's formula \eqref{eq:stanley-derivative} with $G=T$, we have
\[
\frac{\partial X_T}{\partial p_1}=\sum_{v\in V(T)}X_{T-v}.
\]
By Lemma~\ref{lem:pi}, for every graph $G$ we get
\[
\pi X_G=
\sum_{\substack{A\subseteq E(G)\\
                 (V(G),A)\text{ has no isolated vertex}}}
(-1)^{|A|}p_{\lambda_G(A)},
\]
because $\pi$ sends $p_1$ to $0$ and fixes $p_r$ for $r\ge2$, so that
$\pi p_{\lambda_G(A)}$ equals $p_{\lambda_G(A)}$ when no part of
$\lambda_G(A)$ is $1$, that is, when $(V(G),A)$ has no isolated vertex,
and equals $0$ otherwise.  Applying the linear map $\pi$ to the first
identity~(\ref{eq:stanley-derivative}), and then using Lemma~\ref{lem:pi} with $G=T-v$ for each $v$, gives
\begin{equation}\label{eq:phi-leaf-sum}
\Phi_T=\sum_{v\in V(T)}\pi X_{T-v}
=\sum_{v\in V(T)}
\sum_{\substack{A\subseteq E(T-v)\\
                 (V(T-v),A)\text{ has no isolated vertex}}}
(-1)^{|A|}p_{\lambda_{T-v}(A)}.
\end{equation}
Call an edge set $A\subseteq E(T-v)$ \emph{admissible} if $(V(T-v),A)$
has no isolated vertex.  Thus $\Phi_T$ is the sum of the signed monomials
$(-1)^{|A|}p_{\lambda_{T-v}(A)}$ over all vertices $v$ of $T$ and all
admissible $A\subseteq E(T-v)$. We must determine these pairs $(v,A)$.

First let $v=c_i$ be a non-leaf vertex of $T$.  Since $T$ is proper,
$c_i$ is adjacent to a leaf $\ell$ of $T$.  The only edge of $T$ that has $\ell$ as its endpoint is $c_i\ell$, which is removed together with $c_i$, so $\ell$ has
degree $0$ in $T-v$.  Hence $\ell$ is an isolated vertex of
$(V(T-v),A)$ for every $A\subseteq E(T-v)$, no edge set is admissible,
and $\pi X_{T-v}=0$.  Only the leaves $v$ of $T$ remain in
\eqref{eq:phi-leaf-sum}.

The edges of $T$ are the $k-1$ internal edges, which are the edges
$c_ic_j$ with $ij\in E(S)$, and the leaf edges, of which $c_j$ has
$w_j-1$.  The number of leaf edges is therefore
\[
\sum_{j=1}^k(w_j-1)=n-k.
\]
Now let $v$ be a leaf of $T$, adjacent to the non-leaf vertex $c_i$.
Let $L_0$ be the set of leaf edges of $T$ other than $c_iv$, so that
$|L_0|=n-k-1$.  The graph $T-v$ contains these edges and every internal
edge of $T$.  Identifying the latter edges with $E(S)$, every
$A\subseteq E(T-v)$ decomposes uniquely as $A=L\cup F$, where
$L\subseteq L_0$ and $F\subseteq E(S)$.  This classification refers to
edges of the original tree $T$, since an internal edge of $T$ can become
a leaf edge in $T-v$.

Let $A=L\cup F$ be admissible.  If an edge $c_j\ell\in L_0$ were
missing from $L$, then its original leaf endpoint $\ell$, whose only
edge in $T-v$ is $c_j\ell$, would be isolated in $(V(T-v),A)$.
Hence $L=L_0$, and
\[
(-1)^{|A|}=(-1)^{|L|}(-1)^{|F|}=(-1)^{n-k-1}(-1)^{|F|}
=\varepsilon(-1)^{|F|}.
\]
This is the origin of the sign $\varepsilon$.  Next we determine the
partition $\lambda_{T-v}(A)$.  For $C\in\Comp(F)$, the vertices $c_j$
with $j\in C$ are connected to one another by the internal edges in $F$,
and each of them is joined by the edges in $L$ to all of its surviving
leaves from $T$.  Since $L=L_0$, and every vertex of $T-v$ is either
some $c_j$ or a leaf of $T$ other than $v$, the components of
$(V(T-v),A)$ are exactly the sets
\[
\begin{gathered}
\{c_j\mid j\in C\}\cup
\left\{\ell\in V(T)\setminus\{v\}\;\middle|\;
\begin{gathered}
\ell\text{ is a leaf of $T$ adjacent}\\
\text{to some }c_j\text{ with }j\in C
\end{gathered}\right\},\\
C\in\Comp(F).
\end{gathered}
\]
Let $C_i$ be the component of $([k],F)$ containing $i$.  The component
of $(V(T-v),A)$ corresponding to $C$ has order
$\sum_{j\in C}w_j=\mathbf w(C)$ when $C\ne C_i$, and order
$\mathbf w(C_i)-1$ when $C=C_i$, because the deleted leaf $v$ was one of
the $w_i-1$ leaves at $c_i$.  Therefore
\begin{equation}\label{eq:leaf-monomial}
(-1)^{|A|}p_{\lambda_{T-v}(A)}
=\varepsilon(-1)^{|F|}
p_{\mathbf w(C_i)-1}
\prod_{C\in\Comp(F)\setminus\{C_i\}}p_{\mathbf w(C)}.
\end{equation}

Conversely, let $F\subseteq E(S)$ be arbitrary and let $L=L_0$.
Every surviving leaf of $T$ is covered by its edge in $L$, and every
$c_j$ with $j\ne i$ is covered by one of its $w_j-1\ge1$
leaf edges in $L$.  The vertex $c_i$ is covered by a leaf edge in $L$ if
$w_i\ge3$, and by an internal edge in $F$ if $C_i\ne\{i\}$.  So
$A=L\cup F$ fails to be admissible only when $w_i=2$ and
$C_i=\{i\}$, and in exactly this case $\mathbf w(C_i)-1=1$, so that the
monomial in \eqref{eq:leaf-monomial} contains the factor $p_1$.  Every
other index in that monomial is at least $2$, since $w_j\ge2$ for all
$j$.  Hence applying $\pi$ to the right side of \eqref{eq:leaf-monomial}
leaves the admissible terms unchanged and removes the terms that are not
admissible.  Summing over all $F\subseteq E(S)$ gives
\[
\pi X_{T-v}
=\varepsilon\,\pi\left(
\sum_{F\subseteq E(S)}(-1)^{|F|}
p_{\mathbf w(C_i)-1}
\prod_{C\in\Comp(F)\setminus\{C_i\}}p_{\mathbf w(C)}
\right).
\]
This expression depends on $v$ only through the index $i$ of its
neighbour $c_i$, and there are $w_i-1$ leaves adjacent to $c_i$.
Summing over all leaves $v$ of $T$ in \eqref{eq:phi-leaf-sum} therefore
gives
\begin{align*}
\Phi_T
&=\varepsilon\,\pi\left(
\sum_{F\subseteq E(S)}(-1)^{|F|}
\sum_{i=1}^k(w_i-1)p_{\mathbf w(C_i)-1}
\prod_{C\in\Comp(F)\setminus\{C_i\}}p_{\mathbf w(C)}
\right)\\
&=\varepsilon\,\pi\Biggl(
\sum_{F\subseteq E(S)}(-1)^{|F|}
\sum_{C_0\in\Comp(F)}
(\mathbf w(C_0)-|C_0|)p_{\mathbf w(C_0)-1}\\
&\hspace{4em}{}\times
\prod_{C\in\Comp(F)\setminus\{C_0\}}p_{\mathbf w(C)}
\Biggr),
\end{align*}
where the second equality groups the indices $i$ according to the
component $C_0=C_i$ that contains them and uses
\[
\sum_{i\in C_0}(w_i-1)=\mathbf w(C_0)-|C_0|.
\]
This is \eqref{eq:master}.  The forest $([k],F)$ has $k-|F|$
components, so the monomial in the summand for $F$ and $C_0$ has one
power-sum factor for each component, that is, $k-|F|$ factors, and its
degree is
\[
(\mathbf w(C_0)-1)+\sum_{C\in\Comp(F)\setminus\{C_0\}}\mathbf w(C)=n-1.
\]
\end{proof}

For $a\ge2$,
define the \emph{weight multiplicity} $m_a$ by
\[
m_a=\bigl|\{v\in V(S)\mid\mathbf w(v)=a\}\bigr|.
\]
For $2\le a\le b$, define the \emph{edge-type count} $e_{ab}$ by
\[
e_{ab}=\bigl|\{uv\in E(S)\mid
\{\mathbf w(u),\mathbf w(v)\}=\{a,b\}\}\bigr|.
\]
An unordered pair $\{a,b\}$ is \emph{feasible} if $m_a,m_b>0$ when
$a\ne b$, and if $m_a\ge2$ when $a=b$.  Thus the feasible pairs are
exactly the possible endpoint-weight pairs.

The monomial
\[
P_T=\prod_{a\ge2}p_a^{m_a}
\]
records the multiset of skeleton weights.  Define the differential operator
\[
\Delta=\sum_{a\ge3}(a-1)p_{a-1}\frac{\partial}{\partial p_a}.
\]
We order monomials lexicographically with $p_2>p_3>\cdots$.  Let
$\operatorname{LT}(f)$ denote the leading term, including its coefficient,
and set $\operatorname{LT}(0)=0$.  For a monomial $Q$ containing some
$p_a$ with $a\ge3$, define
\[
s(Q)=\min\{a\ge3\mid p_a\text{ divides }Q\}.
\]

\begin{lemma}\label{lem:recoverP}
Let $Q\in\Q[p_2,p_3,\ldots]$ be a power-sum monomial.  If $\Delta Q=0$,
then $Q=p_2^d$ for some $d\ge0$.  Otherwise,
$\operatorname{LT}(\Delta Q)$ uniquely determines $Q$.
\end{lemma}

\begin{proof}
For any monomial $B$, let $u_a(B)$ be the exponent of $p_a$ in $B$.
If $\Delta Q\ne0$, then
\[
\operatorname{LT}(\Delta Q)
=(s(Q)-1)u_{s(Q)}(Q)
Q\frac{p_{s(Q)-1}}{p_{s(Q)}}.
\]
Suppose $\operatorname{LT}(\Delta Q)=\operatorname{LT}(\Delta Q')\ne0$,
and let $s=s(Q)$ and $t=s(Q')$.  Assume without loss of generality that
$s\le t$.  If $s=t$, equality of the leading monomials gives $Q=Q'$.
If $s<t$, comparison of exponents first gives $s=3$ and then
\[
Q=A p_3p_{t-1},
\qquad
Q'=A p_2p_t,
\]
where $A$ contains none of $p_3,\ldots,p_{t-1}$.  For $t=4$, the two
leading coefficients are $4$ and $3(u_4(A)+1)$.  They are unequal.  For
$t\ge5$, they are $2$ and $(t-1)(u_t(A)+1)$, which is at least $4$.
They are again unequal.  Thus the leading term determines $Q$.  Finally,
$\Delta Q=0$ exactly when no $p_a$ with $a\ge3$ occurs.
\end{proof}

For a feasible pair $\{a,b\}$, define
\begin{equation}\label{eq:theta}
P_{T,ab}=\frac{P_T}{p_ap_b},
\qquad
\Theta_{ab}
=
(a+b-2)p_{a+b-1}P_{T,ab}
+p_{a+b}\Delta P_{T,ab}.
\end{equation}

\begin{lemma}\label{lem:theta}
The polynomials $\Theta_{ab}$, indexed by feasible unordered pairs, are
linearly independent over $\Q$.
\end{lemma}

\begin{proof}
Let
\[
M_{ab}=\frac{p_{a+b-1}}{p_ap_b},
\qquad
R_{ab,r}=\frac{p_{a+b}p_{r-1}}{p_ap_bp_r}.
\]
Dividing by $P_T$ gives
\[
\frac{\Theta_{ab}}{P_T}
=(a+b-2)M_{ab}
+\sum_{r\ge3}(r-1)
\bigl(m_r-\one{r=a}-\one{r=b}\bigr)R_{ab,r}.
\]
The Laurent monomials $M_{ab}$ are distinct.  Each $M_{cd}$ has one
power-sum factor in its numerator and two in its denominator, while
$R_{ab,r}$ initially has two numerator factors and three denominator
factors.  Thus $R_{ab,r}$ can equal some $M_{cd}$ only if one numerator
factor cancels one denominator factor.  The possible cancellations are
\[
\begin{aligned}
p_{a+b}=p_r &\quad\Longleftrightarrow\quad r=a+b,
& R_{ab,r}&=M_{ab},\\
p_{r-1}=p_a &\quad\Longleftrightarrow\quad r=a+1,
& R_{ab,r}&=M_{a+1,b},\\
p_{r-1}=p_b &\quad\Longleftrightarrow\quad r=b+1,
& R_{ab,r}&=M_{a,b+1}.
\end{aligned}
\]
Whenever one of these terms has nonzero coefficient, the pair on the right
is feasible because the required factors occur in $P_T$.  Hence a term
$M_{ab}$ contributed by a different polynomial $\Theta_{cd}$ must satisfy
\[
c+d=a+b-1.
\]
Suppose
\[
\sum_{a\le b}c_{ab}\Theta_{ab}=0,
\]
and choose $c_{ab}\ne0$ with $a+b$ minimal.  No different pair can
contribute to $M_{ab}$.  Its coefficient in the divided relation is
\[
\bigl((a+b-2)+(a+b-1)m_{a+b}\bigr)c_{ab},
\]
which is nonzero.  This is a contradiction.
\end{proof}

Let
\[
U=|\{a\ge2\mid m_a>0\}|
\]
be the number of distinct skeleton weights, and let $h$ be the number of
feasible pairs.  We now show that $\Phi_T$ determines the weight
multiplicities and the edge-type counts.  The proof is a recovery
procedure carried out in stages, and we record here what it uses.  When
$h\ge1$, it uses the degree of $\Phi_T$, the coefficient
$[p_{n-1}]\Phi_T$, the leading term of $[\Phi_T]_k$, and $h-1$
coefficients of $[\Phi_T]_{k-1}$ chosen after $P_T$ is known.  When
$h=0$, the degree and $[p_{n-1}]\Phi_T$ suffice.  In addition to the
degree and the monomial position of the leading term when
$[\Phi_T]_k\ne0$, the number of selected coefficient values is at most
\[
1+\binom{U+1}{2}.
\]
Locating the leading term or certifying that the layer is zero may require
inspecting additional coefficients.  Once $P_T$ is known, sparse
representations of the required indices and the resulting triangular
system can be constructed and the system can be solved in $O(U^2)$
arithmetic operations.  These counts are verified at the end of the proof.

\begin{lemma}\label{lemma:recovery}
The polynomial $\Phi_T$ determines $n$, $k$, every $m_a$, and every
$e_{ab}$.
\end{lemma}

\begin{proof}
A term in the master formula has one factor only when $F=E(S)$.  The
skeleton forest then has the single component $[k]$, so
\begin{equation}\label{eq:one}
[p_{n-1}]\Phi_T
=\varepsilon(-1)^{k-1}(n-k)
=(-1)^{n-2}(n-k).
\end{equation}
Since $n-k>0$, the polynomial $\Phi_T$ is nonzero.  The master formula is
homogeneous of degree $n-1$, so
\[
n=1+\deg\Phi_T.
\]
Equation~\eqref{eq:one} then determines $k$.

The $k$-factor layer comes only from $F=\varnothing$.  Choosing vertex $i$
as the distinguished component gives
\[
(w_i-1)p_{w_i-1}\prod_{j\ne i}p_{w_j}.
\]
After applying $\pi$ and summing over $i$, we obtain
\begin{equation}\label{eq:weights}
[\Phi_T]_k
=\varepsilon\,\pi\left(
\sum_{i=1}^k(w_i-1)p_{w_i-1}\prod_{j\ne i}p_{w_j}
\right)
=\varepsilon\Delta P_T.
\end{equation}
Since $n$ and $k$ are known, so is $\varepsilon$, and hence $\Phi_T$
determines $\Delta P_T=\varepsilon[\Phi_T]_k$.  By
Lemma~\ref{lem:recoverP}, we recover $P_T$ from $\Delta P_T$, because
$P_T$ is a power-sum monomial in $p_2,p_3,\ldots$, and the lemma states
that such a monomial $Q$ equals $p_2^d$ for some $d\ge0$ when
$\Delta Q=0$ and is otherwise uniquely determined by
$\operatorname{LT}(\Delta Q)$.  If this layer vanishes, then $P_T=p_2^k$,
since $P_T$ has exactly $k$ factors.  In either case every $m_a$ is
determined.

The $(k-1)$-factor layer comes from sets containing one skeleton edge.
Suppose its endpoint weights are $a$ and $b$.  If the joined component is
distinguished, its contribution is
\[
(a+b-2)p_{a+b-1}P_{T,ab}.
\]
If a remaining singleton component is distinguished, the total contribution
is
\[
p_{a+b}\Delta P_{T,ab}.
\]
A one-edge set has sign $-1$, so summing over the skeleton edges gives
\begin{equation}\label{eq:edges}
[\Phi_T]_{k-1}
=-\varepsilon\sum_{a\le b}e_{ab}\Theta_{ab}.
\end{equation}
Every pair $\{a,b\}$ with $e_{ab}>0$ is feasible, and each $\Theta_{ab}$
depends only on $P_T$, which is already known.  By Lemma~\ref{lem:theta},
we get every $e_{ab}$ from $[\Phi_T]_{k-1}$, because the polynomials
$\Theta_{ab}$ indexed by feasible pairs are linearly independent over
$\Q$, so the coefficients $e_{ab}$ in the expansion \eqref{eq:edges} of
$-\varepsilon[\Phi_T]_{k-1}$ in terms of these polynomials are unique.

Let the feasible pairs be
\[
q_i=(a_i,b_i),
\qquad
1\le i\le h,
\]
ordered by increasing $a_i+b_i$, with ties ordered arbitrarily.  Let
\[
L_i=P_TM_{a_ib_i},
\qquad
A_{ij}=[M_{a_ib_i}]\frac{\Theta_{a_jb_j}}{P_T}.
\]
Each $L_i$ is a power-sum monomial because $q_i$ is feasible.  By the
proof of Lemma~\ref{lem:theta}, we get that $A$ is lower triangular with
\[
A_{ii}=a_i+b_i-2+(a_i+b_i-1)m_{a_i+b_i}>0,
\]
because that proof shows that $M_{a_ib_i}$ occurs in
$\Theta_{a_jb_j}/P_T$ for $j\ne i$ only when
\[
a_i+b_i=a_j+b_j+1,
\]
which forces $j<i$ in the chosen ordering, and that the coefficient of
$M_{a_ib_i}$ in $\Theta_{a_ib_i}/P_T$ is the displayed value.
Equation~\eqref{eq:edges} gives
\[
-\varepsilon[L_i]\Phi_T
=
\sum_{j=1}^h A_{ij}e_{a_jb_j}.
\]
The first $h-1$ equations therefore recover
$e_{a_1b_1},\ldots,e_{a_{h-1}b_{h-1}}$.  The remaining count follows from
\[
e_{a_hb_h}=k-1-\sum_{j=1}^{h-1}e_{a_jb_j}.
\]
There are
\[
h=\binom{U}{2}+|\{a\ge2\mid m_a\ge2\}|
\le\binom{U+1}{2}
\]
feasible pairs.  Thus the procedure uses $[p_{n-1}]\Phi_T$, the leading
coefficient of $[\Phi_T]_k$ when that layer is nonzero, and $[L_i]\Phi_T$
for $1\le i<h$.  This gives at most $1+\binom{U+1}{2}$ selected coefficient values.
The indices in the last group are chosen only after $P_T$ has been recovered.
If $h=0$, then $k=1$, and $n$ is the sole skeleton weight.  Finally, using
the common base $P_T$, each $L_i$ is represented by the exponent changes in
$M_{a_ib_i}$ and therefore takes $O(1)$ data.  The matrix has $h$ rows and
columns and
$O(h)=O(U^2)$ nonzero entries.  It can therefore be constructed and solved
in $O(U^2)$ arithmetic operations once $P_T$ is known.
\end{proof}

\par\medskip
\palm*

\begin{proof}
Let $(S,\mathbf w)$ and $(S',\mathbf w')$ be the weighted skeletons of
$T$ and $T'$.  By Lemma~\ref{lemma:recovery}, we get that $T$ and $T'$
have the same order $n$, the same number $k$ of non-leaf vertices, the
same weight multiplicities $m_a$, and the same edge-type counts $e_{ab}$,
because each of these quantities is determined by $\Phi_T=\Phi_{T'}$.
If $k=1$, both skeletons consist of a single vertex of weight $n$.  If
$k=2$, both consist of a single edge, and the multiset of its two
endpoint weights is given by the $m_a$.  In these cases
$(S,\mathbf w)\cong(S',\mathbf w')$, so assume $k\ge3$.

For each weight $a$, the degrees of the weight-$a$ vertices of a skeleton
sum to
\[
d_a=2e_{aa}+\sum_{b\ne a}e_{\min(a,b),\max(a,b)},
\]
since an edge with both endpoints of weight $a$ adds $2$ to this sum and
an edge with exactly one endpoint of weight $a$ adds $1$.  This value is
the same for $S$ and $S'$.  Suppose $m_a>1$.  Every weight-$a$ vertex of
$S$ is a leaf because $(S,\mathbf w)$ is multiplicity-isolated, so
$d_a=m_a$.  Every vertex of $S'$ has degree at least $1$ because $S'$ is
a tree with $k\ge2$ vertices, so the $m_a$ vertices of weight $a$ in
$S'$ have degrees summing to $m_a$ and are all leaves.  Thus
$(S',\mathbf w')$ is also multiplicity-isolated.

Call a weight $a$ \emph{unique} if $m_a=1$ and \emph{repeated} if
$m_a\ge2$.  In each skeleton every vertex with a repeated weight is a
leaf, so every non-leaf vertex has a unique weight.  A tree with at least
three vertices has no edge joining two leaves, so every edge of $S$ or
$S'$ has a non-leaf endpoint, and hence an endpoint with a unique weight.
In particular, $e_{ab}=0$ when $a$ and $b$ are both repeated, and
$e_{aa}=0$ for every $a$, since a unique weight occurs at only one
vertex.

For a unique weight $a$, let $u_a$ and $u'_a$ be the vertices of weight
$a$ in $S$ and in $S'$.  For unique weights $a<b$, the vertices $u_a$
and $u_b$ are the only vertices of weights $a$ and $b$ in $S$, so they
are adjacent exactly when $e_{ab}=1$, and likewise $u'_a$ and $u'_b$ are
adjacent in $S'$ exactly when $e_{ab}=1$.  For a unique weight $a$ and a
repeated weight $b$, let $L_{ab}$ be the set of weight-$b$ leaves of $S$
adjacent to $u_a$, and let $L'_{ab}$ be the corresponding set in $S'$.
Every edge with endpoint weights $a$ and $b$ joins $u_a$ to a weight-$b$
leaf, so
\[
|L_{ab}|=|L'_{ab}|=e_{\min(a,b),\max(a,b)}.
\]
Every vertex of $S$ is either $u_a$ for a unique weight $a$ or a leaf
with a repeated weight $b$, and such a leaf lies in exactly one of the
sets $L_{ab}$, namely the one indexed by the weight $a$ of its unique
neighbour.  The same holds for $S'$.

Define $\varphi\colon V(S)\to V(S')$ by $\varphi(u_a)=u'_a$ for every
unique weight $a$, and on each $L_{ab}$ by an arbitrary bijection
$L_{ab}\to L'_{ab}$.  Then $\varphi$ is a bijection and
$\mathbf w'(\varphi(x))=\mathbf w(x)$ for every $x\in V(S)$.  An edge of
$S$ either joins $u_a$ to $u_b$ for unique weights $a<b$, in which case
$e_{ab}=1$ and $u'_au'_b\in E(S')$, or joins $u_a$ to a leaf in
$L_{ab}$, whose image lies in $L'_{ab}$ and is therefore adjacent to
$u'_a$.  Thus $\varphi$ maps edges to edges, and since $S$ and $S'$
both have $k-1$ edges, it maps $E(S)$ onto $E(S')$.  Hence $\varphi$ is
an isomorphism of weighted trees.  Attaching $\mathbf w(i)-1$ new leaves
to each vertex $i$ recovers a proper tree from its weighted skeleton, so
$T\cong T'$.
\end{proof}

\section{Uniform Leaf Extensions}
\label{sec:other-results}
In this section we prove Theorem~\ref{thm:Cq}.

For $q\ge2$, let
$\psi_q\colon\Q[p_1,p_2,\ldots]\to\Q[p_1,p_2,\ldots]$ be the ring homomorphism defined by
\[
\psi_q(p_r)=p_{qr}.
\]
Define the linear operator $\mathcal D_q$ by
\[
\mathcal D_qf
=
\pi\left(
\sum_{r\ge1}r p_{qr-1}
\psi_q\left(\frac{\partial f}{\partial p_r}\right)
\right).
\]
The map $\psi_q$ sends distinct power-sum monomials to distinct power-sum
monomials, so it is injective.  The $r$th summand in $\mathcal D_qf$
contains the factor $p_{qr-1}$.  These factors are distinct for different
values of $r$, and none can occur in
\[
\psi_q\left(\frac{\partial f}{\partial p_r}\right),
\]
whose variables have indices divisible by $q$.  Hence terms from different
summands cannot cancel.

If $q\ge3$, none of the factors $p_{qr-1}$ is removed by $\pi$.  Thus
$\mathcal D_qf=0$ implies
\[
\psi_q\left(\frac{\partial f}{\partial p_r}\right)=0
\]
for every $r\ge1$.  Injectivity gives $\partial f/\partial p_r=0$ for
every $r$, so $f\in\Q$.  If $q=2$, only the factor $p_1$ from $r=1$ is
removed.  Thus $\partial f/\partial p_r=0$ for every $r\ge2$, so
$f\in\Q[p_1]$.  Therefore
\begin{equation}\label{eq:Dq-kernel}
\ker\mathcal D_q=
\begin{cases}
\Q,&q\ge3,\\
\Q[p_1],&q=2.
\end{cases}
\end{equation}

Let $G$ have order $t$.  The construction of $C_q(G)$ attaches $(q-1)t$
new leaves.  By \eqref{eq:stanley-derivative} and Lemma~\ref{lem:pi},
$\Phi_{C_q(G)}$ is the sum of the signed power-sum monomials of spanning
subgraphs with no isolated vertex after a vertex deletion.  Deleting an
original vertex leaves one of its attached leaves isolated, so only
deletions of attached leaves occur in this sum.  After one attached leaf
is deleted, all $(q-1)t-1$ remaining edges to attached leaves must be
selected, since omitting any such edge would isolate its leaf endpoint.
Writing $F\subseteq E(G)$ for the selected original edges, the sign of
the resulting edge set is therefore $(-1)^{(q-1)t-1}(-1)^{|F|}$.
If the deleted leaf was attached to a vertex in a component $C$ of
$(V(G),F)$, then there are $|C|$ choices for that vertex and $q-1$
choices for the leaf.  This gives the factor $(q-1)|C|$.  The component
containing the deleted leaf's neighbour has order $q|C|-1$, and every
other component $D$ has order $q|D|$.  Applying $\pi$ removes exactly
the cases with $q|C|-1=1$.  Summing over $F$ and its components gives
\[
\begin{aligned}
\Phi_{C_q(G)}
&=(-1)^{(q-1)t-1}(q-1)
\pi\Biggl(
\sum_{F\subseteq E(G)}(-1)^{|F|}
\sum_{C\in\Comp_G(F)}
|C|p_{q|C|-1}\\
&\hspace{6em}{}\times
\prod_{D\in\Comp_G(F)\setminus\{C\}}p_{q|D|}
\Biggr).
\end{aligned}
\]
Applying $\mathcal D_q$ yields
\begin{equation}\label{eq:Cq-operator}
\Phi_{C_q(G)}
=(-1)^{(q-1)t-1}(q-1)\mathcal D_qX_G.
\end{equation}

\Cq*

\begin{proof}
The reverse implication follows from homogeneity and
\eqref{eq:Cq-operator}.  For the forward implication, suppose the common
$\Phi$ is nonzero.  Its degree gives $|V(G)|=|V(H)|=t$, and
\eqref{eq:Cq-operator} gives
\[
X_G-X_H\in\ker\mathcal D_q.
\]
If $q\ge3$, homogeneity and \eqref{eq:Dq-kernel} give $X_G=X_H$.  If
$q=2$, then $X_G-X_H=cp_1^t$.  Since
$[p_1^t]X_G=[p_1^t]X_H=1$, again $c=0$.

If the common $\Phi$ is zero, \eqref{eq:Dq-kernel} forces $q=2$ and
$X_G\in\Q[p_1]$, and likewise for $H$.  But a connected graph of order
$t>1$ has
\[
[p_2p_1^{t-2}]X_G=-|E(G)|\ne0.
\]
\nopagebreak[4]
Hence $G\cong H\cong K_1$.
\end{proof}

\medskip

\section{Connected-Partition Truncations}
\noindent In this section we prove Theorem~\ref{thm:counting}.

Let $\mathcal T_k$ be the set of isomorphism classes of $k$-vertex trees.
For $S\in\mathcal T_k$ and $F\subseteq E(S)$, the \emph{component type}
of $F$ is the integer partition of $k$ formed by the orders of the
connected components of the spanning forest $(V(S),F)$.
For $\mu\vdash k$, let $N_\mu(S)$ be the number of edge sets
$F\subseteq E(S)$ of component type $\mu$.

A \emph{connected partition} of $S$ is a set partition $\mathcal B$
of $V(S)$ such that the subgraph induced by each block is connected.
Its \emph{type} is the integer partition formed by the block orders.

For each edge set $F\subseteq E(S)$ counted by $N_\mu(S)$, the component
vertex sets of $(V(S),F)$ form a connected partition of type $\mu$.
Conversely, because $S$ is a tree, each connected partition determines
a unique such edge set, consisting of all edges of $S$ whose endpoints
lie in the same block. Thus $N_\mu(S)$ also counts the connected
partitions of $S$ of type $\mu$. If $\mathcal B$ has type $\mu$ and
corresponds to the edge set $F$, then
\[
|F|=\sum_{B\in\mathcal B}(|B|-1)
=k-|\mathcal B|
=k-\ell(\mu).
\]
We call this number the \emph{rank} of the connected partition.

For each nonnegative integer $t$, let
\[
\mathcal P_{k,t}
=\{\mu\vdash k\mid \ell(\mu)\ge k-t\},
\]
and define
\[
\sigma_t\colon\mathcal T_k\longrightarrow
\mathbb Z_{\ge0}^{\mathcal P_{k,t}},
\qquad
\sigma_t(S)=\bigl(N_\mu(S)\bigr)_{\mu\in\mathcal P_{k,t}}.
\]
Thus $\sigma_t(S)$ records the number of connected partitions of each
type having rank at most $t$, or equivalently, the number of edge sets
of each component type among those containing at most $t$ edges.
We call it the \emph{rank-$t$ connected-partition truncation}.
For $0\le s\le t$, $\sigma_s(S)$ is obtained from $\sigma_t(S)$ by
retaining only the entries of rank at most $s$.

Let
\[
\tau(k)=
\min\Bigl(
\{t\in\mathbb Z_{\ge0}\mid
\sigma_t\text{ is injective on }\mathcal T_k\}\cup\{\infty\}
\Bigr).
\]
Thus $\tau(k)$ is the least rank through which the connected-partition
counts distinguish all $k$-vertex trees up to isomorphism, or $\infty$
if no such rank exists.

Since every edge set in a tree is acyclic, Stanley's power-sum expansion
takes the form
\begin{equation}\label{eq:truncation-csf}
X_S
=
\sum_{\mu\vdash k}
(-1)^{k-\ell(\mu)}N_\mu(S)p_\mu.
\end{equation}
For $k\ge2$, the only coordinate of $\sigma_{k-1}(S)$ omitted from
$\sigma_{k-2}(S)$ is $N_{(k)}(S)=1$. Hence $\sigma_{k-2}(S)$ and $X_S$
determine one another. Therefore $\tau(k)<\infty$ exactly when the
chromatic symmetric function distinguishes all trees in $\mathcal T_k$,
and $\tau(k)\le k-2$ whenever it is finite.

The truncation records how many edge sets have each component type,
but not which edge sets give that type.

Write $P_m$ for the path on $m$ vertices.  For $k\ge4$ and
$0\le a\le k-2$, let $T_{k,a}$ be obtained from
\[
P_{k-1}=v_0v_1\cdots v_{k-2}
\]
by attaching a new leaf at $v_a$.  Reflection of the path gives
$T_{k,a}\cong T_{k,k-2-a}$, so the canonical choices are
\[
0\le a\le\left\lfloor\frac{k-2}{2}\right\rfloor.
\]
These canonical trees are pairwise nonisomorphic.  The tree $T_{k,0}$ is
$P_k$.  For $a\ge1$, the unique vertex of degree $3$ has branches of
lengths $1,a,k-2-a$, which determine the canonical value of $a$.

\begin{proposition}\label{prop:path-leaf-family}
Let $k\ge4$ and
\[
0\le a<b\le\left\lfloor\frac{k-2}{2}\right\rfloor.
\]
Then
\[
\sigma_{a+1}(T_{k,a})=\sigma_{a+1}(T_{k,b}).
\]
Moreover,
\[
N_{(a+3,1^{k-a-3})}(T_{k,a})=k-2
\quad\text{and}\quad
N_{(a+3,1^{k-a-3})}(T_{k,b})=k-1.
\]
Hence $a+2$ is the least rank at which their connected-partition counts
differ.
\end{proposition}

\begin{proof}
For a graph $G$, define
\[
Z_G(z)=\sum_{F\subseteq E(G)}z^{|F|}
\prod_{C\in\Comp_G(F)}p_{|C|}.
\]
Thus, when $G$ is a tree,
\[
[z^{|V(G)|-\ell(\mu)}p_\mu]Z_G=N_\mu(G).
\]

A \emph{composition} $\alpha\models m$ is an ordered list of positive
integers with sum $m$.  Its length is $\ell(\alpha)$.  For $m\ge1$, let
\[
A_m=\sum_{\alpha\models m}z^{m-\ell(\alpha)}
p_{\alpha_1}\cdots p_{\alpha_{\ell(\alpha)}}
\]
and
\[
B_m=\sum_{\alpha\models m}z^{m-\ell(\alpha)}
p_{\alpha_1+1}p_{\alpha_2}\cdots p_{\alpha_{\ell(\alpha)}}.
\]
The polynomial $A_m$ enumerates the edge subsets of a path on $m$
vertices.  The polynomial $B_m$ records the result of enlarging the
component at a specified endpoint by one vertex.  The new edge is not
included in its $z$-degree.

We first show that, whenever $1\le l\le r$, the polynomial
\[
B_lA_r-A_lB_r
\]
has no term of $z$-degree less than $l$.  The claim is immediate when
$l=r$, so suppose $l<r$.  A term in either product is indexed by
compositions $\alpha\models l$ and $\beta\models r$.  Its $z$-degree is
\[
d=(l-\ell(\alpha))+(r-\ell(\beta)).
\]
Regard a composition of $m$ as a choice of breaks among $m-1$ positions.
Then $m-\ell(\alpha)$ is the number of missing breaks of $\alpha$.

Suppose $d<l$.  The two compositions must have a common positive partial
sum at most $l$, where the full sum $l$ is allowed for $\alpha$.  Indeed,
if they did not, then at each position from $1$ through $l-1$ at least one
composition would have no break.  The composition $\beta$ would also have
no break at $l$.  These $l$ missing breaks would give $d\ge l$.

Let $s$ be the least common positive partial sum and write
\[
\alpha=\alpha'\alpha'',
\qquad
\beta=\beta'\beta'',
\qquad
|\alpha'|=|\beta'|=s.
\]
The suffix $\alpha''$ may be empty.  Replace $(\alpha,\beta)$ by
\[
(\beta'\alpha'',\alpha'\beta'').
\]
This operation is an involution.  It preserves the $z$-degree and every
power-sum factor, while moving the increment on the first part from the
composition of $l$ to the composition of $r$.  It therefore pairs all terms
of degree less than $l$ in $B_lA_r$ with those in $A_lB_r$.

Now compare $T_{k,c}$ and $T_{k,c+1}$.  If the new leaf edge is omitted,
the same edge subsets occur in both trees.  If that edge and
$v_cv_{c+1}$ are both selected, the new leaf is added to the same path
component in both trees.  When the leaf edge is selected and
$v_cv_{c+1}$ is omitted, the path splits into paths on $c+1$ and
$k-c-2$ vertices.  Therefore
\[
Z_{T_{k,c}}(z)-Z_{T_{k,c+1}}(z)
=z\bigl(B_{c+1}A_{k-c-2}-A_{c+1}B_{k-c-2}\bigr).
\]
For $a\le c<b\le\lfloor(k-2)/2\rfloor$, we have
$c+1\le k-c-2$.  By the preceding argument with $l=c+1$ and
$r=k-c-2$, we get that this difference has no term of degree at most
$c+1$, because $B_lA_r-A_lB_r$ has no term of $z$-degree less than $l$
whenever $1\le l\le r$, and the factor $z$ raises every degree by one.
In particular, it has no term of degree at
most $a+1$.  Telescoping from $c=a$ to $c=b-1$ gives
\[
\sigma_{a+1}(T_{k,a})=\sigma_{a+1}(T_{k,b}).
\]

It remains to count connected $(a+2)$-edge subtrees.  Those avoiding the
new leaf edge are intervals of $a+2$ edges in the original path, of which
there are $k-a-3$.  Those using the new leaf edge are obtained from an
interval of $a+1$ path edges containing the attachment vertex.  Since
$b\ge a+1$ and $k-2-b\ge b$, both sides of $v_b$ contain at least $a+1$
path edges.  The two sides of $v_a$ contain $a$ and at least $a+2$ path
edges.  There are therefore $a+2$ such intervals at $v_b$ and $a+1$ at
$v_a$.  The two total counts are $k-1$ and $k-2$.
\end{proof}

\counting*

\begin{proof}
Taking
\[
b=\left\lfloor\frac{k-2}{2}\right\rfloor
\quad\text{and}\quad
a=b-1,
\]
and using Proposition~\ref{prop:path-leaf-family}, we get
$\sigma_b(T_{k,b-1})=\sigma_b(T_{k,b})$. The trees
$T_{k,b-1}$ and $T_{k,b}$ are nonisomorphic, so $\sigma_b$ is not
injective on $\mathcal T_k$.  Since $\sigma_s$ is obtained from
$\sigma_b$ by discarding entries whenever $s\le b$, no $\sigma_s$ with
$s\le b$ is injective either.  Hence
\begin{equation}\label{eq:tau-linear}
\tau(k)\ge b+1=\left\lfloor\frac{k}{2}\right\rfloor.
\end{equation}

This proves the first assertion. It remains to prove the second.

Fix $t$ and assume $k>t$.  Let $p(j)$ denote the number
of partitions of $j$, with $p(0)=1$.  For $0\le j\le t$, subtracting $1$
from every part and discarding any zero parts injects the partitions of $k$
with $k-j$ parts into the partitions of $j$.  The original partition is
recovered by adding $1$ to every remaining part and then adjoining enough
parts equal to $1$ to obtain $k-j$ parts.  There are therefore at most
$p(j)$ coordinates in the $j$-edge layer.  Each coordinate lies between
$0$ and $\binom{k-1}{j}$.  Hence the number $M_{k,t}$ of possible
truncations satisfies
\begin{equation}\label{eq:M}
M_{k,t}
\le
\prod_{j=0}^t
\left(\binom{k-1}{j}+1\right)^{p(j)}
\le
(2k^t)^{P(t)},
\qquad
P(t)=\sum_{j=0}^t p(j).
\end{equation}

Let $s_{k,t}$ be the number of trees in singleton fibers of $\sigma_t$.
Distinct such trees have distinct truncations, so
\[
s_{k,t}\le M_{k,t}.
\]
For fixed $t$, \eqref{eq:M} gives
\[
M_{k,t}\le(2k^t)^{P(t)}=k^{O_t(1)}.
\]
Otter's theorem~\cite{Otter1948} states that the number of unlabeled,
unrooted trees on $k$ vertices satisfies
\[
|\mathcal T_k|\sim C\alpha^k k^{-5/2}
\qquad(k\to\infty),
\]
for constants $C>0$ and $\alpha>1$. In particular,
$|\mathcal T_k|=\exp(\Theta(k))$.
Since $M_{k,t}$ grows at most polynomially in $k$ for fixed $t$ and
$s_{k,t}\le M_{k,t}$, we conclude that
\[
0\le
\frac{s_{k,t}}{|\mathcal T_k|}
\le
\frac{M_{k,t}}{|\mathcal T_k|}
\longrightarrow0
\qquad(k\to\infty).
\]
\end{proof}

\section{Lean formalization}\label{sec:formalization}

An accompanying Lean development is available in the author's
repository~\cite{AwanLean2026}. It was produced by the author with
assistance from Aristotle. The archived version is commit
\texttt{2827a80aba46}, and uses Lean~4.28.0 and Mathlib~4.28.0.
The main declarations are collected in
\path{RequestProject/MainTheorems.lean}. They state the reconstruction
theorem for multiplicity-isolated proper trees, the equivalence for
uniform leaf extensions, the bound $\tau(k)\ge\lfloor k/2\rfloor$,
and the density-zero conclusion for singleton fibers of a fixed
truncation.

The development represents the power-sum variables by independent
variables in a multivariate polynomial ring over $\Q$, and defines
$X_G$ using Stanley's expansion~\eqref{eq:stanley}. Thus the passage
from proper colorings to that expansion is supplied by
Stanley~\cite{Stanley1995}; it is not a separate result of the Lean
development. A source review found
no \texttt{sorry} placeholders or project-specific axiom declarations.
An independent clean build of the archived development has not been
performed in preparing this version.

\section*{Acknowledgements}
My study of Stanley's chromatic symmetric function problem began in a
directed reading course at the University of Kansas in August 2024 and
developed into an undergraduate honors thesis, completed in May 2025.
The present paper grew out of that work. I am deeply grateful to my
thesis advisor, Dr. Jeremy Martin, for his guidance, encouragement, and
generous support throughout this project. I also thank Dr. Reuven Hodges
for his careful reading of the manuscript and for generously providing
extensive and detailed comments on very short notice.

\ifjocformat
\bibliographystyle{imsart-number}
\else
\bibliographystyle{amsalpha}
\fi
\bibliography{references}

\end{document}